\documentclass[oneside]{amsart}
\usepackage[utf8]{inputenc}
\usepackage{amssymb,graphicx,hyperref,algorithm}
\usepackage[scale=0.7]{geometry}
\usepackage[noEnd=true]{algpseudocodex}
\usepackage{url}

\newtheorem{lemma}{Lemma}
\newtheorem{theorem}[lemma]{Theorem}

\newtheorem{prob}[lemma]{Question}

\theoremstyle{remark}

\newcommand{\R}{\mathbb{R}}
\newcommand{\abs}[1]{\lvert #1 \rvert}
\newcommand{\intprod}[1]{\left\langle #1 \right\rangle}

\title{On the orthogonal Grünbaum partition problem in dimension three}

\author[G.L. Maldonado]{Gerardo L. Maldonado}
\address[G.L. Maldonado]{Centro de Ciencias Matemáticas, UNAM Campus Morelia, Morelia, Mexico}
\email{gmaldonado@matmor.unam.mx}

\author[E. Roldán-Pensado]{Edgardo Roldán-Pensado}
\address[E. Roldán-Pensado]{Centro de Ciencias Matemáticas, UNAM Campus Morelia, Morelia, Mexico}
\email{e.roldan@im.unam.mx}


\begin{document}
	
	\begin{abstract}
		Grünbaum's equipartition problem asked if for any measure $\mu$ on $\R^d$ there are always $d$ hyperplanes which divide $\R^d$ into $2^d$ $\mu$-equal parts. This problem is known to have a positive answer for $d\le 3$ and a negative one for $d\ge 5$. A variant of this question is to require the hyperplanes to be mutually orthogonal. This variant is known to have a positive answer for $d\le 2$ and there is reason to expect it to have a negative answer for $d\ge 3$. In this note we exhibit measures that prove this.
		Additionally, we describe an algorithm that checks if a set of $8n$ in $\R^3$ can be split evenly by $3$ mutually orthogonal planes. To our surprise, it seems the probability that a random set of $8$ points chosen uniformly and independently in the unit cube does not admit such a partition is less than $0.001$.
	\end{abstract}
	
	\maketitle
	
	\section{Introduction}
	Let $\mu$ be a finite Borel measure on $\R^d$, we say that $\mu$ is a mass if it vanishes on every affine hyperplane.	
	A collection of hyperplanes $H_1,H_2,\dots,H_k\subset\R^d$ define an \emph{equipartition} of a mass $\mu$ if these hyperplanes divide $\R^d$ into $2^k$ parts of equal $\mu$-measure. The problem of determining the values of $d$ for which any mass has an equipartition given by $d$ hyperplanes was proposed by Grünbaum in 1960 \cite{Gru60} and is a special case of the Grünbaum-Hadwiger-Ramos mass partition problem \cite{RS2022,BFHZ2018}.
	Grünbaum's problem has an affirmative answer for $d\le 3$ \cite{Had66,YDEP89}, a negative one for $d\ge 5$ \cite{Avi84} and is unsolved for $d=4$.
	
	We are interested in a variant of Grünbaum's problem where an additional condition is imposed, namely, that the hyperplanes defining the equipartition be mutually orthogonal. In this way, we say that the hyperplanes $H_1,H_2,\dots,H_k\subset\R^d$ define an \emph{orthogonal equipartition} of $\mu$ if these hyperplanes are mutually orthogonal and divide $\R^d$ into $2^k$ parts of equal $\mu$-measure.
	\begin{prob}\label{prob:measure}
		Is it true that every mass $\mu$ on $\R^d$ has an orthogonal equipartition defined by $d$ hyperplanes?
	\end{prob}
	
	The purpose of this paper is to answer this question negatively for $d\ge 3$. It is well known that the answer is affirmative for $d\le 2$ and, thanks to Avis' work \cite{Avi84}, we know it to be negative for $d\ge 5$ even if the hyperplanes are not required to be pairwise orthogonal.
	\begin{theorem}\label{thm:measure}
		For every $d\ge 3$ there exists a mass on $\R^d$ that has no orthogonal equipartition defined by $d$ planes.
	\end{theorem}
	
	From a topological point of view, this result is expected. This is because the space of triples consisting of mutually orthogonal planes in $\R^3$ has dimension $6$ but, in order for such a triple to define an orthogonal equipartition of $\mu$, $7$ independent constraints must be satisfied. Since the number of constraints is larger that the dimension of the space, no triple of mutually orthogonal planes should satisfy all the constraints for some \emph{generic} measure $\mu$.
	
	Some of the main ideas used to prove Theorem \ref{thm:measure} are similar to the ones used by Avis. Both in his and our approach, the measures constructed are supported on the moment curve. In Section \ref{sec:pre} we present the proof of Theorem \ref{thm:measure} along with the necessary tools and definitions. We also briefly discuss the implications of our results for the orthogonal version of the Grünbaum-Hadwiger-Ramos problem.
	
	There are other interesting results involving mass partitions with orthogonal hyperplanes. For example, if $d\in\{2,4,8\}$, then any two masses in $\R^d$ can be simultaneously bisected by $d-1$ mutually orthogonal hyperplanes \cite{Sim2018}.
	
	Additionally, we examine an equivalent problem in which the measures are replaced by point sets. Given a set of $8n$ points in $\R^3$, an \emph{orthogonal equipartition} of $X$ is a partition of $X$ into $8$ sets consisting of $n$ points each which can be separated by using $3$ mutually orthogonal planes.	
	In Section \ref{sec:algo} we describe an algorithm that is capable of finding all orthogonal equipartitions $X$.
	\begin{theorem} \label{thm:time}
		Let $X$ be a set of $8n$ points in $\R^3$, then the set of all orthogonal equipartitions of $X$ can be found in time $O(n^7)$.
	\end{theorem}
	There is a previous algorithm, described in \cite{Dia1991}, that runs in time $O(n^{15/2}\log(n))$.
	We believe the time complexity in our algorithm can be further reduced to $O(n^5\log(n))$. For the $2$-dimensional version of this problems there is an algorithm that finds all equipartitions in time $O(n\log(n))$ \cite{RS2007}.
	
	A surprising thing that we found is almost all random sets of points we explored have orthogonal equipartitions defined by three planes. We considered sets of either $8$ or $16$ random points chosen uniformly and independently in a cube. Of these only a fraction of around $0.0009$ sets of $8$ points and of $0.0007$ sets of $16$ points failed to have such an equipartition. This came as a surprise to us in view of the topological intuition described above.
	
	There has been recent interest in equipartitions of point sets. In \cite{ABTRW2024}, it is shown that any set of points in $R^3$ admits an equipartition (under a slightly different definition) by $3$ planes, such that the intersection of two of the planes is a line with a prescribed direction. Moreover, they present an efficient algorithm that computes such an equipartition for a set of $n$ points, given the prescribed direction. This is achieved in time $O^*(n^{5/2})$, where the $O^*$-notation hides polylogarithmic factors.
	
	\section{Orthogonal hyperplanes and the moment curve}\label{sec:pre}
	
	The standard moment curve in $\R^d$ is defined parametrically by the mapping
	\[\gamma(t)=\left(t,t^2,\dots,t^d\right)\]
	for $t\in\R$.
	
	In order to prove Theorem \ref{thm:measure} we require a few lemmas. The first one is probably well-known but we include a short proof. In simple words it states that not all elements in an orthonormal basis can be too far from any given direction.
	
	\begin{lemma}\label{lem:raiz_3}
		Let $\{u_1,\dots,u_d\}$ be an orthonormal basis of $\R^d$ and let $v\in\R^d$ be a unit vector. Then there is an $i\in\{1,\dots,d\}$ such that $\abs{\intprod{u_i,v}}\ge 1/\sqrt{d}$.
	\end{lemma}
	\begin{proof}
		Since $\{u_1,\dots,u_d\}$ is an orthonormal basis we have that $v=\intprod{u_1,v} u_1 + \dots + \intprod{u_d,v} u_d$, so $1=\intprod{v,v}=\intprod{u_1,v}^2+\dots+\intprod{u_d,v}^2$. Let $i\in\{1,\dots,d\}$ be the index for which $\intprod{u_i,v}^2$ is maximal, then $d \intprod{u_i,v}^2\ge \intprod{u_1,v}^2+\dots+\intprod{u_d,v}^2 = 1$ which implies that $\abs{\intprod{u_1,v}}\ge 1/\sqrt{d}$.
	\end{proof}
	
	For large $M$, the part of the moment curve with $t\ge M$ behaves like a vertical line since its last coordinate dominates the rest. This is exploited in the final lemma. 
	
	\begin{lemma}\label{lem:tail}
		For every $d\ge 2$ there is a number $M_d>0$ such that the following holds: If $H_1,H_2,\dots, H_d$ are mutually orthogonal hyperplanes in $\R^d$, then the part of the moment curve defined by
		\[\Gamma=\{\gamma(t):t\ge M_d\}\]
		intersects some $H_i$ in at most one point.
	\end{lemma}
	\begin{proof}
		Let $u_i$ be the vector normal to $H_i$ with non-negative last coordinate. Let $e_d=(0,\dots,0,1)\in \R^d$, then Lemma \ref{lem:raiz_3} implies that one of the $u_i$, say $u_1$, satisfies
		\[\intprod{u_1,e_d}\ge \frac{1}{\sqrt d}.\]
		
		Assume that $H_1$ intersects $\Gamma$ at two distinct points $\gamma(a)$ and $\gamma(b)$ with $0<a,b$.
		Then the vector $v = \frac 1{b-a}(\gamma(b) - \gamma(a))$ is parallel to $H_1$ and therefore $\intprod{u_1,v}=0$.
		Note that
		\[v = \frac{\gamma(b)-\gamma(a)}{b-a} = (1,S_1,S_2,\dots,S_{d-1}),\]
		where $S_n = a^n+a^{n-1}b+\dots+ab^{n-1}+b^n$.
		
		If $u_1=(x_1,\dots,x_d)$ then $x_d=\intprod{u_1,e_d}\ge 1/\sqrt d$ and therefore
		\begin{align*}
			0 & = \intprod{u_1,v} = x_1 + x_2 S_1+\dots + x_d S_{d-1}\\
			& > x_1 - S_1-\dots - S_{d-2} + \frac{S_{d-1}}{\sqrt{d}},
		\end{align*}
		so
		\begin{equation}\label{eq:}
			x_1 < S_1 + \dots + S_{d-2} - \frac{S_{d-1}}{\sqrt{d}}.
		\end{equation}
		The right-hand side of this inequality is a polynomial on $a$ and $b$, and satisfies that the coefficients multiplying the terms of maximal degree are negative. Therefore, if $M_d$ is large enough and $M_d\le a,b$, then \eqref{eq:} implies $x_1<-1$ which contradicts $\|u_1\|=1$.
	\end{proof}
	
	For $d=3$ it is not hard to see that we may choose $M_3<2$. On the other hand, there does exist a triple of mutually orthogonal planes such that each of them intersects $\{\gamma(t):t\ge 0\}$ in at least $2$ points.
	
	Now we are ready to prove out main theorem.
	
	\begin{proof}[Proof of Theorem \ref{thm:measure}]
		Let $M_d$ and $\Gamma$ be as in Lemma \ref{lem:tail} and let $\mu$ be any mass supported on $\Gamma$.
		Assume that $\mu$ has an orthogonal equipartition defined by the planes $H_1,\dots, H_d$, then there must be a segment of positive length of $\Gamma$ in each of the $2^d$ connected components of $\R^d\setminus(H_1\cup\dots\cup H_d)$. This implies that $\Gamma$ must intersect $H_1\cup\dots\cup H_d$ at least $2^d-1$ times.
		
		By Lemma \ref{lem:tail} there is a plane, say $H_1$, which $\Gamma$ intersects only once. Since $\Gamma$ is a subset of the moment curve, $\Gamma$ intersects every $H_i$, with $i>1$ in at most $d$ points. Therefore, the number of intersection points between $\Gamma$ and $H_1\cup\dots\cup H_d$ is at most $1+d(d-1)$ which, for $d\ge 4$, is less than $2^d-1$. This proves that no collection of $d$ mutually orthogonal hyperplanes defines an equipartition of $\mu$ when $d\ge 4$.
		
		For $d=3$ we have that $d(d-1)+1 = 7 = 2^d-1$ so there is more to be done. Below we show that in this case there is a second plane which intersects $\Gamma$ in at most $2$ points.
		
		Assume that a plane $H$ with normal vector $u$ intersects $\Gamma$ at three points $\gamma(a)$, $\gamma(b)$ and $\gamma(c)$ with $0 < a, b, c$.
		Then the vector $v = \frac{1}{(a-b)(a-c)(b-c)}(\gamma(b)-\gamma(a)) \times (\gamma(c)-\gamma(a))$ is orthogonal to $H$ and therefore parallel to $u$.
		Note that
		\[v = (1,-a-b-c,ab+ac+bc).\]
		Let $y=a+b+c$ and $z=ab+ac+bc$, then clearly $0<y,z$.
		
		If $\Gamma$ intersects $H_2$ and $H_3$ in at least $3$ points each, their normal vectors $u_2$ and $u_3$ are parallel to vectors of the form $(1,-y_2,z_2)$ and $(1,-y_3,z_3)$ with $0<y_2,z_2,y_3,z_3$.
		Since $H_2$ and $H_3$ are orthogonal,
		\[0=\intprod{(1,-y_2,z_2),(1,-y_3,z_3)}=1+y_2y_3+z_2z_3 > 1,\]
		which is a contradiction.
		Thus, either $H_2$ or $H_3$ intersects $\Gamma$ in at most $2$ points.
		Therefore, the number of intersection points between $\Gamma$ and $H_1,H_2,H_3$ is at most $6$, which is less than $2^3-1$, proving that no collection of $3$ mutually orthogonal planes defines an equipartition of $\mu$ when $d=3$.
	\end{proof}
	
	We showed that, for $d=2$, $\Gamma$ intersects any pair of orthogonal lines in at most $3$ points. For $d=3$, $\Gamma$ intersects any triple of mutually orthogonal planes in at most $6$ points. It might be possible that, for arbitrary $d$, $\Gamma$ intersects any $k$-tuple of hyperplanes in at most
	\begin{equation}\label{eq:numpoints}
		d+(d-1)+\dots+(d-k+1)=\frac {k(2d-k+1)}{2}
	\end{equation}
	points, however we have not attempted to verify this.
	
	The Grünbaum-Hadwiger-Ramos problem asks to determine the smallest number $d=\Delta(j,k)$ such that for any $j$ masses in $\R^d$, there exist $k$ hyperplanes that simultaneously equipartition each of the masses. A natural variant of this problem introduces the additional constraint that the $k$ hyperplanes must be mutually orthogonal, we denote the corresponding number by $\Delta^\perp(j,k)$.
	
	Similar to the argument in \cite{Ram96}, if we place $j$ masses supported on $\Gamma$ with mutually disjoint supports, then a simultaneous equipartition of the $j$ masses by $k$ hyperplanes would need to cut the support of each mass in at least $2^k-1$ points, for a total of $j(2^k-1)$ points. On the other hand, by Lemma \ref{lem:tail}, the $k$ hyperplanes intersect $\Gamma$ in at most $(k-1)d+1$ points. It follows that $(k-1)d+1\ge j(2^k-1)$ and therefore
	\begin{equation*}
		\Delta^\perp(j,k) \ge \frac{j (2^k - 1)-1}{k-1}.
	\end{equation*}
	However, this bound can be improved in some cases, as was done in the proof of Theorem \ref{thm:measure} for $j=1$ and $k=3$. If \eqref{eq:numpoints} holds, we would have $k(2d-k+1)/2\ge j(2^k-1)$ and thus
	\begin{equation*}
		\Delta^\perp(j,k) \ge \frac{j (2^k - 1)}{k}+\frac{k-1}2.
	\end{equation*}
	
	\section{Discrete partition problem}\label{sec:algo}
	
	In this section we study a discrete version of Question \ref{prob:measure} where we consider point-sets instead of measures.
    
    We say that a set of points in $\R^3$ is in \emph{general position} if, from $X$, no $3$ points are colinear, no $4$ points are coplanar, no $6$ points are in the union of $2$ orthogonal planes, and no $7$ points are in the union of $3$ orthogonal planes. Note that this is what is expected for sets of points generated randomly, so this can be achieved by perturbing any given set of points. It is easy to see that, if $\abs{X}\ge 7$, it is enough to ask for the condition that no $7$ points are in the union of $3$ orthogonal planes in order to satisfy the rest.
    
    Recall that an oriented plane $H\subset\R^3$ divides $\R^3$ into two open half-spaces called the positive and negative sides, we denote them by $H^+$ and $H^-$, respectively.
    Suppose we have a triple of intersecting oriented planes $H_1,H_2,H_3$, then for any choice of signs $s_1,s_2,s_3\in\{+,-\}$ we have a closed region defined by
    \begin{equation}\label{eq:regions}
    	\bigcap_{i\in\{1,2,3\}} \overline{H_i^{s_i}}.
    \end{equation}
    Assume that $X$ is a set with $8n$ points in $\R^3$, and the sets $X_1,\dots,X_8$ form a partition of $X$.
    We say that this partition is an \emph{orthogonal equipartition} of $X$ if each $X_i$ has $n$ points and there exist three orthogonal oriented planes $H_1,H_2,H_3$ together with bijections $X_i\mapsto R_i$ which assigns to each $X_i$ one of the $8$ regions $R_i$ defined in \eqref{eq:regions} such that $X_i\subset R_i$. In this case we say that the planes $H_1,H_2,H_3$ \emph{certify} the equipartition.
    
    Our algorithm is based on the following key observation that has a standard and straightforward proof.
    \begin{lemma}\label{lem:6}
	    If $X$ has an orthogonal equipartition certified by the planes $H_1,H_2,H_3$, then there are planes $H_1',H_2',H_3'$ that also certify this partition and whose union contains exactly $6$ points from $X$.	
    \end{lemma}
	
	In the opposite direction, it is possible to show that one may always perturb the planes in order to obtain a triple of mutually orthogonal planes containing no points from $X$ which certify the same equipartition.
	
	Assume that $\abs{H_1\cap X}\ge \abs{H_2\cap X}\ge \abs{H_3\cap X}$, then there are only two possibilities: either $\abs{H_1\cap X}=3$, $\abs{H_2\cap X}=2$ and $\abs{H_3\cap X}=1$, or $\abs{H_1\cap X}=\abs{H_2\cap X}=\abs{H_3\cap X}=2$.
	
   	Let $A_1, A_2, A_3\subset\R^3$ be finite sets such that $A_1\cup A_2\cup A_3$ is in general position. If $\abs{A_1}=3$, $\abs{A_2}=2$ and $\abs{A_1}=1$ then it is easy to see that there is a unique triple of orthogonal planes $H_1$, $H_2$ and $H_3$ containing $A_1$, $A_2$ and $A_3$, respectively. Computing these planes is also straightforward. If $A_1=A_2=A_3=2$ then there is at most one triple of orthogonal planes $H_1$, $H_2$ and $H_3$ containing $A_1$, $A_2$ and $A_3$, respectively. This triple can be computed in the following way.

	First, for each $i\in\{1,2,3\}$, consider the vectors $v_i$ which are the difference between the two points in $A_i$. The planes $H_1,H_2,H_3$ should be mutually orthogonal and satisfy that $v_i\in H_i$. If $u_i$ is a unit normal vector to $H_i$ then this is equivalent to $u_1,u_2,u_3$ being an orthonormal basis satisfying $u_i \perp v_i$.
    	
    Choose $a,b$ as two independent vectors, whose linear span does not contain any of the $v_i$. Take a basis of $v_i^\perp$, for example $a_{i} = a\times v_i$ and $b_{i} = b\times v_i$. In order for $u_i\perp v_i$, each $u_i$ should then be a linear combination of $a_i$ and $b_i$. To simplify, assume that the coefficient of $b_i$ is $1$ and let
    \[u_i = \alpha_i a_i + b_i.\]
    For these vectors to be orthogonal we require, for each $i\in\{1,2,3\}$,
	\begin{align}
		0=u_i\cdot u_{i+1}&=\alpha_i\alpha_{i+1} (a_i\cdot a_{i+1}) +\alpha_i (a_i\cdot b_{i+1})+\alpha_{i+1} (b_i\cdot a_{i+1}) + (b_i\cdot a_{i+1}) \label{eq:linear}\\
		&=A_i \alpha_i\alpha_{i+1} +B_i\alpha_i+C_i\alpha_{i+1} + D_i, \nonumber
	\end{align}
    where the indices are taken modulo $3$ and
   	\begin{align*}
   		A_i = a_i\cdot a_{i+1}, \qquad B_i = a_i\cdot b_{i+1}, \qquad C_i = b_i\cdot a_{i+1}, \qquad D_i = b_i\cdot a_{i+1},
   	\end{align*}
    This is a system of three quadratic equations whose solutions are given by
    \begin{equation}\label{eq:alpha}
    	\alpha_i = \frac{r_i\pm\sqrt{q}}{2s_i},
    \end{equation}
    where
    \begin{align*}
   		q & = (A_1 B_3 D_2-A_1 C_2 D_3+A_2 B_1 D_3+A_2 C_3 D_1\\
   		&\quad\qquad\qquad -A_3 B_2 D_1+A_3 C_1 D_2-B_1 B_2 B_3-C_1 C_2 C_3)^2 \\
   		&\quad +4 (A_1 A_3 D_2-A_1 C_2 C_3+A_2 B_1 C_3-A_3 B_1 B_2) \\
   		&\quad\qquad\qquad  (-A_2 D_1 D_3+B_2 B_3 D_1-B_3 C_1 D_2+C_1 C_2 D_3), \\
   		r_i & = A_i B_{i+2} D_{i+1}-A_i C_{i+1} D_{i+2}+A_{i+1} B_i D_{i+2}+A_{i+1} C_{i+2} D_i\\
   		&\quad-A_{i+2} B_{i+1} D_i+A_{i+2} C_i D_{i+1}-B_i B_{i+1} B_{i+2}-C_i C_{i+1} C_{i+2}, \\
   		s_i & = - A_i A_{i+2} D_{i+1} + A_i C_{i+1} C_{i+2} - A_{i+1} B_i C_{i+2} + A_{i+2} B_i B_{i+1}).
   	\end{align*}
   	If $q<0$, there is no real solution and therefore there does not exist a triple of planes $H_1,H_2,H_3$ satisfying the desired properties.
   	If $s_i=0$, we may not assume that $u_i$ is of the form $\alpha_ia_i+b_i$.
   	In other words, any solution would have $u_i$ parallel to $a_i$.
   	In this case we may assume that $u_i=a_i$ in order to construct and solve a system of equations similar to \eqref{eq:linear}.
   	
   	Once we have found the vectors $u_1,u_2,u_3$ one only needs to construct the planes $u_1^\perp,u_2^\perp,u_3^\perp$ and translate them appropriately in order th determine $H_1,H_2,H_3$.
   	Because of the choice of the sign in \ref{eq:alpha}, there are $8$ possible real solutions, these correspond to triples of vectors of the form $\pm u_1,\pm u_2,\pm u_3$ which all give rise to the same triple of planes.
   	
   	For our computational goals it is important to notice that the value of $q$ in \eqref{eq:alpha} does not depend on the index, so the parameters defining the planes $H_1,H_2,H_3$ have coefficients in $\mathbb Q[\sqrt{q}]$. This means that, if $q\ge 0$, we can determine precisely whether any given point lies on the positive or negative side of a plane. 

    \begin{algorithm}
    	\caption{Finding an orthogonal equipartition in dimension $3$.}\label{alg:cap}
    	\begin{algorithmic}
    		\Require A set $X\subset\R^3$ of $8n$ points with integer coordinates in general position
    		\Ensure The set $E$ of orthogonal equipartitions of $X$
    		\State $E \gets \emptyset$
    		\ForAll{$A_1\in \binom{X}{3}$, $A_2\in \binom{X}{2}$, $A_3\in \binom{X}{1}$}
	    		\State $H_1,H_2,H_3\gets$ \Call{planes\_321}{$A_1$,$A_2$,$A_3$}
	    		\ForAll{$p\in $ \Call{partitions}{$X,H_1,H_2,H_3$}}
		    		\If{$p$ defines an equipartition}
			    		\State Add $p$ to $E$
		    		\EndIf
	    		\EndFor
    		\EndFor
    		\ForAll{distinct $A_1\in \binom{X}{2}$, $A_2\in \binom{X}{2}$, $A_3\in \binom{X}{2}$}
	    		\State $H_1,H_2,H_3\gets$ \Call{planes\_222}{$A_1$,$A_2$,$A_3$}
	    		\If{$H_1$,$H_2$,$H_3$ are valid}
		    		\ForAll{$p\in $ \Call{partitions}{$X,H_1,H_2,H_3$}}
			    		\If{$p$ is an equipartition}
				    		\State Add $p$ to $E$
			    		\EndIf
		    		\EndFor
	    		\EndIf
    		\EndFor
    	\end{algorithmic}
    \end{algorithm}
    
	Now we are ready to describe Algorithm \ref{alg:cap}. It receives as input a set of $8n$ points in general position.
	First it looks at all possible triples $A_1,A_2,A_3\subset X$ with $\abs{A_1}=3$, $\abs{A_2}=2$ and $\abs{A_3}=1$ and uses the function \textsc{planes\_321} to construct the triple of orthogonal planes $H_1,H_2,H_3$ satisfying $H_1\cap X=A_1$, $H_2\cap X=A_2$ and $H_3\cap X=A_1$. The function \textsc{PARTITIONS} then constructs all possible partitions of $X$ certified by these planes by assigning the possible regions to each to the $6$ points from $X$ in $H_1\cup H_2\cup H_3$. If all parts have the same size, the partition is added to the list of orthogonal equipartitions.
	
	Afterwards it looks at the triples $A_1,A_2,A_3\subset X$ with $\abs{A_1}=\abs{A_2}=\abs{A_3}=2$ and uses the function \textsc{PLANES\_222} to construct, whenever they exist, the triple of orthogonal planes $H_1,H_2,H_3$ satisfying $H_1\cap X=A_1$, $H_2\cap X=A_2$ and $H_3\cap X=A_1$. If the planes do exist then the algorithm proceeds as above.
	
	We are left with the set of all possible orthogonal equipartition of $X$.

	Our algorithm runs in time $O(n^7)$ since it takes $O(n^6)$ time to iterate over the sets $A_1,A_2,A_3$ and for each of these it requiers another $O(n)$ to check if the partition is an equipartition.
	We believe that this can be improved to $O(n^{9/2} \log(n))$ since, in all sets $X$ of points we tested, whenever $X$ had an equipartition it could be certified with planes satisfying $\abs{H_1\cap X}=3$, $\abs{H_2\cap X}=2$ and $\abs{H_3\cap X}=1$.
	If only these planes were needed to verify that a set $X$ has no orthogonal equipartition then we could iterate only over the sets $A_1\subset X$ with $\abs{A_1}=3$ to find the planes $H_1\supset A_1$ which halve $X$ (there are at most $n^{5/2}$ \cite{SST2001}), project $X$ onto $H_1$ and solve the $2$-dimensional version of the problem on $H_1$ (time $O(n\log(n))$ \cite{RS2007}) which produces planes $H_2,H_3$ orthogonal and orthogonal to $H_3$. This is almost, but not quite, an equipartition, so we can check in time $O(n)$ if it is.
	
	A version of our algorithm has been implemented using SageMath \cite{sagemath}.The reader can find it in \href{https://github.com/XGEu2X/Grunbaum-3d-Partitions}{https://github.com/XGEu2X/Grunbaum-3d-Partitions}.
	The repository includes instructions on how to use it, our experiments and a couple of examples of point sets which both possess and lack orthogonal equipartitions.
   
	\section{Acknowledgments}
	We extend our heartfelt gratitude to Jeffrey Uhlmann for his invaluable discussions and for pointing out the wider interest this problem holds. We also thank the anonymous referees, whose insightful suggestions helped improve the paper, especially the proof of Lemma \ref{lem:raiz_3}.
	
	This work was supported by UNAM-PAPIIT project IN111923.
	
	\bibliographystyle{amsalpha}
	\bibliography{main}

\newcommand{\etalchar}[1]{$^{#1}$}
\providecommand{\bysame}{\leavevmode\hbox to3em{\hrulefill}\thinspace}
\providecommand{\MR}{\relax\ifhmode\unskip\space\fi MR }
\providecommand{\MRhref}[2]{%
  \href{http://www.ams.org/mathscinet-getitem?mr=#1}{#2}
}
\providecommand{\href}[2]{#2}
\begin{thebibliography}{ABT{\etalchar{+}}24}

\bibitem[ABT{\etalchar{+}}24]{ABTRW2024}
B.~Aronov, A.~Basit, G.~Tasinato, I.~Ramesh, and U.~Wagner,
  \emph{Eight-partitioning points in {3D}, and efficiently too}, arXiv preprint
  arXiv:2403.02627 (2024).

\bibitem[Avi84]{Avi84}
D.~Avis, \emph{Non-partitionable point sets}, Inform. Process. Lett \textbf{19}
  (1984), no.~3, 125--129.

\bibitem[BFHZ18]{BFHZ2018}
P.~Blagojevi{\'c}, F.~Frick, A.~Haase, and G.~Ziegler, \emph{Topology of the
  {G}r{\"u}nbaum--{H}adwiger--{R}amos hyperplane mass partition problem},
  Transactions of the American Mathematical Society \textbf{370} (2018),
  no.~10, 6795--6824.

\bibitem[Dia91]{Dia1991}
M.~G. Diaz, \emph{Algorithms for balanced partitioning of polygons and point
  sets}, Ph.D. thesis, Johns Hopkins University, 1991.

\bibitem[FFYP89]{YDEP89}
H.~Edelsbrunner F.~F.~Yao, D. P.~Dobkin and M.~S. Paterson, \emph{Partitioning
  space for range queries}, SIAM J. Comput. \textbf{18} (1989), no.~2,
  371--384.

\bibitem[Gr{\"{u}}60]{Gru60}
B.~Gr{\"{u}}nbaum, \emph{Partitions of mass-distributions and of convex bodies
  by hyperplanes}, Pacific J. Math. \textbf{10} (1960), 1257--1261.

\bibitem[Had66]{Had66}
H.~Hadwiger, \emph{Simultane {V}ierteilung zweier {K}örper ({G}erman)}, Arch.
  Math. (Basel) \textbf{17} (1966), 274--278.

\bibitem[Ram96]{Ram96}
E.~A. Ramos, \emph{Equipartition of mass distributions by hyperplanes},
  Discrete \& Computational Geometry \textbf{15} (1996), 147--167.

\bibitem[RPS22]{RS2022}
E.~Rold{\'a}n-Pensado and P.~Sober{\'o}n, \emph{A survey of mass partitions},
  Bulletin of the American Mathematical Society \textbf{59} (2022), no.~2,
  227--267.

\bibitem[RS07]{RS2007}
S.~Roy and W.~Steiger, \emph{Some combinatorial and algorithmic applications of
  the {B}orsuk--{U}lam theorem}, Graphs and Combinatorics \textbf{23} (2007),
  331--341.

\bibitem[S{\etalchar{+}}22]{sagemath}
W.~A. Stein et~al., \emph{{S}age {M}athematics {S}oftware ({V}ersion 9.5)}, The
  Sage Development Team, 2022, {\tt http://www.sagemath.org}.

\bibitem[Sim18]{Sim2018}
S.~Simon, \emph{Mass partitions via equivariant sections of {S}tiefel bundles},
  Filomat \textbf{32} (2018), no.~11, 3759--3768.

\bibitem[SST01]{SST2001}
M.~Sharir, S.~Smorodinsky, and G.~Tardos, \emph{An improved bound for k-sets in
  three dimensions}, Discrete \& Computational Geometry \textbf{26} (2001),
  195--204.

\end{thebibliography}
\end{document}